\documentclass[reqno,12pt]{amsart}
\usepackage{amsmath,amsthm,enumerate,cite,graphics,pstricks,amsfonts,latexsym,amsopn,verbatim,amscd,amssymb}
\usepackage{color}
\usepackage{psfrag}
\usepackage[all]{xy}

\theoremstyle{plain}
\newtheorem{thm}{Theorem}[section]

\newtheorem{lem}[thm]{Lemma}
\newtheorem{cor}[thm]{Corollary}
\newtheorem{prop}[thm]{Proposition}
\newtheorem{conj}[thm]{Conjecture}

\theoremstyle{definition}

\newtheorem{rem}[thm]{Remark}

\DeclareMathOperator{\Aut}{Aut}

\DeclareMathOperator{\Cent}{Cent}

\begin{document} 

\title[On finite CG-groups]{On finite CG-groups} 

\author[S. J. Baishya  ]{Sekhar Jyoti Baishya} 
\address{S. J. Baishya, Department of Mathematics, Pandit Deendayal Upadhyaya Adarsha Mahavidyalaya, Behali, Biswanath-784184, Assam, India.}

\email{sekharnehu@yahoo.com}

\begin{abstract}
A finite group $G$ is a CG-group if  $\mid \Cent(G) \mid =\mid G' \mid+2$, where $G'$ is the commutator subgroup and $\Cent(G)$ is the set of distinct element centralizers of $G$. In this paper we give some results on CG-groups. We also give a negative answer to \cite[Conjecture 2.3]{con} given by K. Khoramshahi and M. Zarrin.
\end{abstract}

\maketitle

\section{Introduction} \label{S:intro}

Given a group $G$, let $\Cent(G)$ denote the set of centralizers of $G$, i.e., $\Cent(G)=\lbrace C(x) \mid x \in G\rbrace $, where $C(x)$ is the centralizer of the element $x$ in $G$. The study of finite groups in terms of $|\Cent(G)|$, becomes an interesting research topic in last few years. Starting with Belcastro and Sherman \cite{ctc092} in 1994  many authors have been studied and characterised finite groups $G$ in terms of $\mid \Cent(G)\mid$. More information on this and related concepts may be found in  \cite{ed09, amiri2019, amiri20191, amiri17, rostami, en09, ctc09, ctc091, ctc099, baishya, baishya1, baishya2, zarrin0941, zarrin0942, non}. 

While studying the number of distinct element centralizers of a finite group $G$, we observe that there are lots of examples of finite groups (including family of finite groups) for which $|\Cent(G)|=\mid G' \mid+2$, where $G'$ is the commutator subgroup of $G$ (the smallest example of such group being $S_3$). We have also observed that such groups have many interesting properties (for example, in most of the cases $G'$ is abelian, proper element centralizers are abelian etc). This motivates us to introduce the notion of CG-group with a hope that this might give some new solvability criterian for a finite group. We call a finite group $G$ to be a CG-group if $\mid \Cent(G) \mid =\mid G' \mid+2$. In this paper, we give some examples of CG-groups and some necessary and sufficient conditions for a finite group to be a CG-group. Apart from these, we also give a negative answer to the conjecture \cite[Conjecture 2.3]{con} given by K. Khoramshahi and M. Zarrin.

In this paper, all groups are finite (however  Lemma \ref{rem144} and Remark \ref{rem1} holds for any group) and all notations are usual. For example $G'$, $Z(G)$ denotes the commutator subgroup  and the center of a group $G$ respectively, $C_n$ denotes the cyclic group of order $n$,  $D_{2n}$ denotes the dihedral group of order $2n$, $A_n$ denotes the alternating group of degree $n$, $S_n$ denotes the symmetric group of degree $n$ and $ C_n {\rtimes}_\theta C_p$ denotes semidirect product of $C_n$ and $C_p$, where $\theta : C_p \longrightarrow \Aut(C_n)$ is a homomorphism.

 \section{Some Examples and basic results}  

We begin with some examples and counter examples of finite CG-groups. As we have already mentioned, the smallest CG-group is $S_3$ (in fact in view of \cite[Corollary 2.5]{baishya}, any group with central quotient of order $pq$ is a CG-group, where $p, q$ are primes, not necessarily distinct). The next higher order CG-groups are $D_8$ and $Q_8$. We will see that any non-abelian group order $pqr$ or $p^4$, where $p, q, r$ are primes not necessarily distinct is a CG-group.
On the other hand, in view of \cite[Theorem 7]{ctc092}, no finite perfect group is a CG-group (recall that a group $G$ is said to be perfect if $G'=G$). Consequently, no finite simple, semisimple group (a group is said to be semisimple if it is a direct product of non-abelian simple groups) is a CG-group and $A_4$ is the only CG-group among the alternating groups $A_n, n\geq 3$. We now give some more examples of finite CG-groups:

\begin{prop}\label{EX111}
The generalized dihedral group $D(m, n)= \langle a, b \mid a^m=b^n=1, bab^{-1}=a^{-1} \rangle$; $m \geq 3, n \geq 2$ and $n$ even is a CG-group. 
\end{prop}

\begin{proof}
By  \cite[Fact 2]{gd}, we have $Z(D(m, n))=\langle a^{\frac{m}{2}}, b^2\rangle$ or $\langle  b^2 \rangle$ according as $m$ is even or odd. Therefore $\langle \langle a \rangle, Z(D(m, n)) \rangle$ is an abelian normal subgroup of $D(m, n)$ of prime index. Now, the result follows using \cite[Theorem 2.3]{baishya}.
\end{proof}

As an immediate corollary we have the following result:

\begin{cor}\label{EX1}
The dihedral group $D_{2n}, n \geq 3$ is a CG-group. 
\end{cor}

\begin{prop}\label{EX2}
Let $\frac{G}{Z(G)} \cong C_n {\rtimes}_\theta C_p$ be non-abelian, $n$ be an integer, $p$ be a prime and $\theta : C_p \longrightarrow \Aut(C_n)$ is a non-trivial homomorphism. Then $G$ is a CG-group.
\end{prop}

\begin{proof}
See  \cite[Proposition 2.9]{baishya}.
\end{proof}

\begin{prop}\label{EX12}
The generalized quaternion group $Q_{4m}=\langle a, b \mid a^{2m}=1, b^2=a^m, bab^{-1}=a^{-1} \rangle, m \geq 2 $ is a CG-group. 
\end{prop}

\begin{proof}
It is well known that $\frac{Q_{4m}}{Z(Q_{4m})} \cong D_{2m}$, for any $m \geq 2$. Therefore by  Proposition \ref{EX2}, $G$ is a CG-group. 
\end{proof}

Let $I(G)$ be the set of all solutions of the equation $x^2=1$ in $G$. Define $\alpha(G)=\frac{\mid I(G)\mid}{\mid G \mid}$. Then we have the following result:

\begin{prop}\label{EX3}
Let $G$ be a finite group of order $2^nm$, where $n$ is any integer and $m>1$ is an odd integer. If $\alpha(G)> \frac{1}{2}$, then $G$ is a CG-group. 
\end{prop}

\begin{proof}
See  \cite[Corollary 2.6] {baishya}.
\end{proof}

\begin{prop}\label{EX5}
The semidihedral or quasidihedral group $SD_{2^n}=\langle x, y \mid x^{2^{n-1}}=y^2=1, yxy^{-1}=x^{2^{n-2}-1}\rangle$ is a CG-group for any $n \geq 3$. 
\end{prop}

\begin{proof}
It follows from \cite[Theorem 2.3]{baishya}, noting that  $SD_{2^n}, n \geq 3$ has an abelian normal subgroup of prime index.
\end{proof}

\begin{prop}\label{EX6}
The modular $p$-group $Mod_n(p)=\langle x, y \mid x^{p^{n-1}}=y^p=1, yxy^{-1}=x^{1+p^{n-2}}\rangle$ is a CG-group for any $n \geq 3$. 
\end{prop}

\begin{proof}
It follows from \cite[Theorem 2.3]{baishya}, noting that  $Mod_n(p), n \geq 3$ has an abelian normal subgroup of prime index.
\end{proof}

\begin{prop}\label{EX7}
The group $U_{6n}=\langle x, y \mid x^{2n}=y^3=1, x^{-1}yx=y^{-1}\rangle$ is a CG-group for any $n$. 
\end{prop}

\begin{proof}
It follows from Proposition \ref{EX2}, noting that $\frac{U_{6n}}{Z(U_{6n})} \cong S_3$ (see \cite[Lemma 2.7]{ctc09}).  
\end{proof}

The linear groups plays a significant role in the theory of finite groups. We now study the  CG-groups among the linear groups of degree two. In the following results, the groups $GL(2, q), SL(2, q), PGL(2, q)$ and $PSL(2, q)$ denote the general linear, special linear, projective general linear and projective special linear groups respectively of degree $2$ over the field of size $q$, where $q$ is a prime power.

\begin{prop}\label{EX9}
The group $G=GL(2, q)$ is a CG-group if and only if $q=2$. 
\end{prop}

\begin{proof}
If $q=2$, then $G=S_3$ is a CG-group. Next, suppose $q> 2$. In view of \cite[Proposition 3.26]{abc}), the proper centralizers of $G$ are precisely the members of the family $\lbrace xDx^{-1}, xIx^{-1}, xPZ(G)x^{-1} \mid x \in G\rbrace$, where
\begin{enumerate}
	\item  $D$ is the subgroup of all diagonal matrices in $G$, and the number of conjugates of $D$ in $G$ is $\frac{q(q+1)}{2}$,
	\item  $I$ is a cyclic subgroup of $G$, and the number of conjugates of $I$ in $G$ is $\frac{q(q-1)}{2}$,
	\item  $P$ is the Sylow $p$-subgroup of $G$ consisting of all upper triangular matrices with $1$ in the diagonal, and the number of conjugates of $PZ(G)$ in $G$ is $q+1$.
\end{enumerate} 
Therefore  $\mid Cent(G) \mid= \frac{q(q+1)}{2}+\frac{q(q-1)}{2}+(q+1)+1=q^2+q+2$. Also by \cite[Theorem 3.1.19]{PGL}, for $q>2$,  we have $G'=SL(2, q)$ having order $q(q-1)(q+1)$. In the present scenario, one can easily verify that $G$ is not a CG-group. 
\end{proof}

\begin{prop}\label{EX22}
The group $G=SL(2, q)$ is a CG-group if and only if $q=2$ or $3$. 
\end{prop}

\begin{proof}
 It is easy to verify that  $SL(2, 2)$ and $SL(2, 3)$ are CG-groups.   On the otherhand, if $q> 3$, then by  \cite{ija}, we have $G=G'$ and consequently using \cite[Theorem 7]{ctc092}, $G$ is not a CG-group. 
\end{proof}

We now compute the number of distinct element centralizers of $PGL(2, q)$.

\begin{prop}\label{EX25}
Let  $G=PGL(2, q)$.  Then 
\[
        \mid \Cent(G) \mid=
        \begin{cases}
        
        	  2q^2+q+2  &\;\text{if \, $q$ \, is odd, $q>3$}\\
        	  
        	  q^2+q+2  &\;\text{if \, $q$ \, is even, $q > 4$}\\

        	  5  &\;\text{if \, $q=2$}\\
        	  
        	  14  &\;\text{if \, $q=3$}\\
        	  
        	  22  &\;\text{if \, $q=4$}
        	          	
        \end{cases}
\]

\end{prop}

\begin{proof}

If $q$ is even, then $PGL(2, q)=PSL(2, q)$ (see \cite{msuzuki}). Now, the  result follows from \cite[Theorem 1.1]{zarrin094}.

Now, suppose $q$ is odd. We have $G=PGL(2, 3)=S_4$ and $\mid \Cent(G) \mid=14$. Next, suppose $q\geq 5$. In view of \cite[Page 5]{hassani}, it follows that the centralizer of any involution (element of order $2$) in $G$  is either $D_{2(q-1)}$ or $D_{2(q+1)}$. Hence by \cite[Proposition 2.4]{abcd}, $\mid \Cent(PGL(2, 5)) \mid=57$, noting that the size of element centralizers in $PGL(2, 5)=S_5$ are $4, 5, 6, 8$ and $12$. On the other hand, if $q>5$, then using  \cite[Table 4]{nichols}, it follows that the size of the proper element centralizers of $G$ are $q, q-1, q+1, 2(q-1)$ and $2(q+1)$.  Again, let $a$ and $b$ be any two distinct involutions in $G$. Suppose $C(a)=C(b)$. Since $C(a), C(b) \in \lbrace D_{2(q-1)}, D_{2(q+1)} \rbrace$, therefore $a=b$ noting that $Z(C(a))=Z(C(b))$, which is a contradiction. Therefore $C(a) \neq C(b)$. In the present scenario, in view of  \cite[Proposition 2.4]{abcd}, 
$\mid Cent(G) \mid= (q+1)+q(q+1)+q(q-1)+1=2q^2+q+2$.
\end{proof}

As an immediate corollary, we have the following results, noting that $(PGL(2, q))'=PSL(2, q)$ for $q \geq 3$ (see \cite{msuzuki}):

\begin{cor}\label{EX227}
 $PGL(2, q)$ is a CG-group if and only if $q=2$ or $3$. 
\end{cor}

\begin{cor}\label{EX2223}
 $S_5$ is not a CG-group. 
\end{cor}

For $PSL(2, q)$, we have the following result:

\begin{prop}\label{EX228}
 $PSL(2, q)$ is a CG-group if and only if $q=2$ or $3$.
\end{prop}

\begin{proof}
For $q>3$ we have $PSL(2, q)$  is a simple group (see \cite{msuzuki}) and hence by \cite[Theorem 7]{ctc092}, $G$ is not a CG-group. On the otherhand $PSL(2, 2)=S_3$ and $PSL(2, 3)=A_4$ are CG-groups.
\end{proof}

We conclude the section with the following lemmas:

\begin{lem}\label{CG20}
Let $G$ be a finite group such that $\mid G' \cap Z(G)\mid =1$. Then $G$ is a CG-group if and only if $\frac{G}{Z(G)}$ is a CG-group. 
\end{lem}

\begin{proof}
The proof follows using \cite[Lemma 3.1]{en09}, noting that if $\mid G' \cap Z(G)\mid =1$, then $\mid (\frac{G}{Z(G)})' \mid= \mid G' \mid$.
\end{proof}

\begin{lem}\label{CG22}
Let $G$ be a finite group such that all sylow subgroups are abelian. Then  $G$ is a CG-group if and only if $\frac{G}{Z(G)}$ is a CG-group.  
\end{lem}

\begin{proof}
In the present scenario, we have $\mid G' \cap Z(G)\mid =1$ (see \cite[p. 118]{suzuki1}). Therefore the result follows from the previous lemma.
\end{proof}

 \section{The main results}

In this section, we prove the main results of the paper. However, we begin with the following lemma.

\begin{lem}\label{rem144}
 Let $G$ be any group. If $\mid \frac{G}{Z(G)} \mid=pqr$ where $p, q, r$ are primes (not necessarily distinct), then $C(x)$ is abelian for any $x \in G \setminus Z(G)$. 
\end{lem}

\begin{proof}
Let $x \in G \setminus Z(G)$. If $\frac{C(x)}{Z(G)}$ is cyclic, then $C(x)$ is abelian. Now, suppose $\mid \frac{C(x)}{Z(G)} \mid=pq$. Then $o(xZ(G))=p, q$ or $pq$. If $o(xZ(G))=pq$, then $C(x)$ is abelian. Next suppose $o(xZ(G))\neq pq$. Then there exists some $y \in C(x)$ such that 
$\frac{C(x)}{Z(G)}= \langle xZ(G), yZ(G) \rangle$. Consequently, $C(x)=\langle x, y, Z(G) \rangle$ and hence $C(x)$ is abelian. If $\mid \frac{C(x)}{Z(G)} \mid=pr$ or $qr$, then using similar arguments we can show that $C(x)$ is abelian.
\end{proof}

\begin{rem}\label{rem1}
Recall that a group $G$ is said to be a CA-group if $C(x)$ is abelian for any $x \in G \setminus Z(G)$. It is easy to see that for such groups $C(x) \cap C(y)=Z(G)$ for any two distinct proper centralizers $C(x)$ and $C(y)$. Also we have seen that if $\mid \frac{G}{Z(G)} \mid=pqr$ where $p, q, r$ are primes (not necessarily distinct), then $G$ is a CA-group. 
\end{rem}

The authors in \cite[Proposition 3.6]{con} showed that if $G$ is a finite group such that $\frac{G}{Z(G)}$ is isomorphic to a simple group, then then $G$ and $G'$ are isoclinic groups.
In the following we generalize this result as follows:

\begin{prop}\label{CG118}
Let $G$ be a finite group such that $\frac{G}{Z(G)}$ is perfect. Then $G$ and $G'$ are isoclinic groups.
\end{prop}

\begin{proof}
Suppose $\frac{G}{Z(G)}$ is perfect. Then $\frac{G}{Z(G)}=(\frac{G}{Z(G)})'=\frac{G'Z(G)}{Z(G)}$ and consequently, $G'Z(G)=G$. Therefore in view of \cite[Lemma 2.7]{pL95}, $G$ is isoclinic to $G'$.
\end{proof}

As an immediate consequence we have the following necessary condition for a finite CG-group.

\begin{prop}\label{CG111}
Let $G$ be a finite group such that $\frac{G}{Z(G)}$ is perfect. Then $G$ is  not a CG-group.
\end{prop}

\begin{proof}
Suppose $\frac{G}{Z(G)}$ is perfect. Then by Proposition \ref{CG118}, $G$ is isoclinic to $G'$ and consequently, using \cite[Lemma 2.3]{non}, we have $\mid \Cent(G) \mid =\mid \Cent(G') \mid$. Now, if $G$ is a CG-group, then $\mid \Cent(G) \mid =\mid \Cent(G') \mid=\mid G' \mid+2$, which is impossible by  \cite[Theorem 7]{ctc092}.
\end{proof}

The following result follows using technique similar to \cite[Lemma 2.7]{en09}.

\begin{prop}\label{CG10}
Let $p$ be the smallest prime divisor of the order of a group $G$ and $\mid G' \mid=p$. Then  $G$ is  a CG-group if and only if $\frac{G}{Z(G)} \cong C_p \times C_p$. 
\end{prop}

As an application to this we have the following result. Given a group $G$, $\omega(G)$ denotes the size of a  maximal set of pairwise non-commuting elements of $G$. 

\begin{prop}\label{CG191}
Let $G$ be a finite non-abelian metacyclic $p$-group, where $p>2$ is a prime. Then $G$ is a CG-group if and only if $\frac{G}{Z(G)} \cong C_p \times C_p$.  
\end{prop}

\begin{proof}
In view of \cite{fouladi5}, we have $\omega(G)=\frac{\mid G'\mid}{p}(1+p)$. Now, suppose $G$ is a CG-group. Then $\mid G' \mid+2 \geq \omega(G)+1=\frac{\mid G'\mid}{p}(1+p)+1$. Consequently, we have $\mid G' \mid=p$. Now, the result follows using Proposition \ref{CG10}.
\end{proof}

The following proposition gives a sufficient condition for a finite group to be a CG-group.

\begin{prop}\label{CG7}
Let $G$ be a finite non-abelian group with an abelian normal subgroup of prime index. Then
\begin{enumerate}
	\item  $G$ is a CG-group.
	\item  If $\frac{G}{Z(G)}$ is abelian, then $\frac{G}{Z(G)}$ is elementary abelian. 
	\item  If $\frac{G}{Z(G)}$ is non-abelian, then $\frac{G}{Z(G)}$ a CG-group. In particular, if $\frac{G}{Z(G)}$ is of order $p^r$ for some prime $p$, then $\mid Cent(G) \mid= p^{r-1}+2$.
\end{enumerate}.
\end{prop}

\begin{proof}
a) See \cite[Theorem 2.3]{baishya}. It may be mentioned here that \cite[Theorem 3.4]{fouladi}) is a particular case of this result, where the author obtained the result for $p$-groups ($p$ a prime )only, noting that $\mid Cent(G) \mid=\omega(G)+1$ if and only if $G$ is a CA-group (\cite[Lemma 2.6]{ed09}).

b) By \cite[Theorem A]{ctc095}, $G$ is a CA-group and consequently, $\lbrace \frac{C(x)}{Z(G)} / x \in G \setminus Z(G)\rbrace$ is a partition of $\frac{G}{Z(G)}$ (see \cite[Remark 2.1]{ed09}). Therefore if  $\frac{G}{Z(G)}$ is abelian, by \cite[ p. 571]{zappa}, we have $\frac{G}{Z(G)}$ is elementary abelian.

c) Suppose $\frac{G}{Z(G)}$ is non-abelian. Let $N$ be an abelian normal subgroup of $G$ of prime index. Then $N=C(x)$ for some $x \in G \setminus Z(G)$. In the present scenerio, we have $\frac{C(x)}{Z(G)}=C(xZ(G))$, and consequently, $C(xZ(G))$ is an abelian normal subgroup of $\frac{G}{Z(G)}$  of prime index.  Therefore by \cite[Theorem 2.3]{baishya}, $\frac{G}{Z(G)}$ is a CG-group.

Again, suppose  $\frac{G}{Z(G)}$ is non-abelian of order $p^r$ for some prime $p$. By \cite[Theorem A]{ctc095}, $G$ is a CA-group. Therefore $C(x) \cap C(y) = Z(G)$ for any $x,y \in G \setminus Z(G), xy \neq yx$. Now,  $N=C(x)$ is a  centralizer of $G$ of index $p$.  Clearly $C(x)$ will contain exactly $p^{r-1}$ distinct right cosets of $Z(G)$. Therefore the number of right cosets of $Z(G)$ (other than $Z(G)$) left for the remaining proper centralizers is $p^r-p^{r-1}=p^{r-1}(p-1)$. In the present scenario, one can verify that any proper centralizer other than $C(x)$ will contain exactly $p$ distinct right cosets of $Z(G)$. Hence $\mid Cent(G) \mid= p^{r-1}+2$.
\end{proof}

As an application of the above proposition, we have the following result for minimal non-abelian group. Recall that a minimal non-abelian group is a non-abelian group all of whose proper subgroups are abelian. By \cite[Aufgaben III. 5.14]{huppert}, we have if $G$ is a finite  minimal non-abelian group, then $G$ can have at the most two distinct prime divisors and  if $G$ is not a prime power group, then $G=PQ$, where $P$ is a cyclic $p$-Sylow subgroup of $G$ and $Q$ is the elementary abelian minimal normal $q$-Sylow subgroup of $G$.

\begin{prop}\label{CG6}
Let $G$ be a finite minimal non-abelian group.  Then
\begin{enumerate}
	\item  $G$ is a CG-group. In particular, (see \cite[Proposition 2.2]{rostami}) if $G$ is a $p$-group ($p$ a prime), then we have $\mid Cent(G) \mid= p+2$. Otherwise, $G$ is a primitive $\mid Q \mid+2$ centralizer group, where $Q$ is the normal Sylow subgroup of $G$.
	\item  If $G$ is a $p$-group ($p$ a prime), then $\frac{G}{Z(G)} \cong C_p \times C_p$. Otherwise, $\frac{G}{Z(G)}$ is a CG-group.
\end{enumerate}.
  
\end{prop}

\begin{proof}
a) Let $G$ be a finite minimal non-abelian group. If $G$ is a $p$-group, for some prime $p$, then $G$ has an abelian normal subgroup of prime index and hence by Proposition \ref{CG7}, $G$ is a CG-group.

Next, suppose $G$ is not a $p$-group. Let $P$ be a cyclic Sylow subgroup and $Q$ be the elementary abelian Sylow subgroup of $G$ respectively. Then $P$ has a normal subgroup $H$ of prime index and consequently, $QH$ is an abelian normal subgroup of $G$ of prime index. Therefore by Proposition \ref{CG7}, $G$ is a CG-group. 

Now, if $G$ is a $p$-group, then $\mid Cent(G) \mid= p+2$, noting that we have $\mid G' \mid=p$ by \cite[Lemma 2.5]{niketora}. On the otherhand, in view of \cite[Lemma 3.1]{en09}, we have $G$ is a primitive $\mid Q \mid+2$ centralizer group, where $Q$ is the normal Sylow subgroup of $G$, niting that in the present scenario, we have $G'=Q$ and $\mid G' \cap Z(G) \mid =1$.

b) If $G$ is a $p$-group, then the result follows from Peoposition \ref{CG10}, noting that we have $\mid G' \mid=p$ by \cite[Lemma 2.5]{niketora}. On the other hand, if $G$ is not a prime power group, then $\frac{G}{Z(G)}$ is not a prime power group, noting that in the present scenario, we have $Z(G) \subsetneq P$, where $P$ is a cyclic Sylow subgroup of $G$. Now, the result follows from Proposition \ref{CG7}.
\end{proof}

Recall that a Frobenius group $G$ is said to be minimal if no proper subgroup of $G$ is Frobenius. In this connection we have the following result:

\begin{prop}\label{CG1}
Let $G$ be a  Frobenius group with kernel $K$ and complement $H$. 
\begin{enumerate}
	\item  If $G$ is minimal Frobenius, then $G$ is a CG-group, 
	\item  If $K$ is cyclic, then $G$ is a CG-group,
	\item  If $H$ is abelian, then $G$ is a CG-group if and only if $G'$ is abelian.
\end{enumerate}.
\end{prop}

\begin{proof}
a) If $G$ is a minimal Frobenius group, then by \cite[Theorem 3.2.14]{perumal},  $K$ is elementary abelian and $H$ has prime order. Again, by \cite[Theorem 2.2]{herzog}, $G'=K$. Therefore from the definition of Frobenius group, $G$ is a CG-group.

b) If $K$ is cyclic, then in view of \cite{Fcyclic}, $H$ is cyclic. In the present scenario, by \cite[Theorem 2.2]{herzog}, we have $G'=K$. Therefore from the definition of Frobenius group, $G$ is a CG-group. 

c) If $H$ is abelian, then by \cite[Theorem 2.2]{herzog}, we have $G'=K$. Now, if $G$ is a CG-group, then $G'$ must be abelian.
Again, if $G'$ is abelian, then from the definition of Frobenius group, $G$ is a CG-group.
\end{proof}

\begin{rem}\label{Remark3}
If $G$ is a finite solvable group in which centralizers of non-identity elements are abelian, then $G$ is a Frobenius group with abelian kernel and cyclic complement (see \cite[Proposition 3.1.1, Proposition 1.2.4]{elizabeth}). In this connection we have the following result:
\end{rem}

\begin{prop}\label{CG3}
Let $G$ be a finite group such that $C(x)$ is abelian for every non-identity element $x \in G$. Then  $G$ is a CG-group if and only if $G$ is a Frobenius group with abelian kernel and cyclic complement. 
\end{prop}

\begin{proof}
Suppose $G$ is a CG-group. If $G$ is  non-solvable, then by \cite[Lemma 3.9]{abc}, $G$ is simple  and consequently using \cite[Theorem 7]{ctc092}, $G$ is not a CG-group. Therefore $G$ is solvable. Now, the result follows from Remars \ref{Remark3}. 

Conversely, suppose $G$ is a Frobenius group with abelian kernel $K$ and cyclic complement. By   \cite[Theorem 2.2]{herzog}, we have $G'=K$. Therefore from the definition of Frobenius group, $G$ is a CG-group.  
\end{proof}

As an immediate consequence we have the following result:

\begin{prop}\label{CG4}
If $G$ is a non-abelian group of order $pqr$, $p,q,r$  being primes (not necessarily distinct), then $G$ is a CG-group.
\end{prop}

\begin{proof}
If $\mid Z(G) \mid =1$, then in view of Remark \ref{rem1}, $G$ is a solvable CA-group and therefore by Proposition \ref{CG3}, $G$ is a CG-group. Again, if $\mid Z(G) \mid \neq 1$, then $\mid \frac{G}{Z(G)} \mid $ is a product of two primes and hence by \cite[Corollary 2.5]{baishya}, $G$ is a CG-group. 
\end{proof}

\begin{prop}\label{CG5}
If $G$ is a non-abelian group of order $p^4$, $p$  being prime, then $G$ is a CG-group.
\end{prop}

\begin{proof}
It is well known that if $G$ is a non-abelian group of order $p^4$, $p$  being prime, then $G$ has an abelian normal subgroup of prime index and hence by Proposition \ref{CG7}, $G$ is a CG-group. 
\end{proof}

We need the following result to prove our next proposition:

\begin{prop}{(Lemma 12.12 \cite{imisaacs})}\label{isaacs}
Let $G$ be a finite group. If $A \vartriangleleft G$  with $A$ abelian and $G/A$ cyclic, then $\mid A \mid= \mid G' \mid \mid A \cap Z(G) \mid$.   
\end{prop}

We now give another sufficient condition for a finite group to be a CG-group.

\begin{prop}\label{CG17}
Let $G$ be a finite group such that $\frac{G}{Z(G)}=\frac{K}{Z(G)} \rtimes \frac{H}{Z(G)}$ is a Frobenius group with $K$ and $H$ abelian. Then $G$ is a CG-group.  
\end{prop}

\begin{proof}
Using the third isomorphic theorem, we get $\frac{G}{K} \cong \frac{H}{Z(G)}$. Consequently,   we have $K$ is an abelian normal subgroup of $G$ such that $\frac{G}{K}$ is cyclic. In the present scenario, in view of Proposition \ref{isaacs}, $\mid K \mid=\mid G' \mid \mid K \cap Z(G) \mid$ which forces $\mid \frac{K}{Z(G)}\mid=\mid G' \mid$. Therefore by \cite[Proposition 3.1]{amiri2019}, $G$ is a CG-group.
\end{proof}

As a consequence we have the following corollaries:

\begin{cor}\label{CG17cor}
Let $G$ be a finite group such that $\frac{G}{Z(G)}=\frac{K}{Z(G)} \rtimes \frac{H}{Z(G)}$ is a Frobenius group with $H$ abelian. If $G'$ is abelian, then $G$ is a CG-group.  
\end{cor}

\begin{proof}
Note that $H$ is abelian implies $\frac{H}{Z(G)}$ is cyclic and consequently, using \cite[Theorem 2.2]{herzog}, we have $(\frac{G}{Z(G)})'=\frac{G'Z(G)}{Z(G)}=\frac{K}{Z(G)}$. Now, if $G'$ is abelian, then $G'Z(G)$ is abelian and therefore, by Proposition \ref{CG17}, $G$ is a CG-group. 
\end{proof}

\begin{cor}\label{CG24}
Let $G$ be a finite group such that $\frac{G}{Z(G)}$ is minimal Frobenius group. If $G'$ is abelian, then $G$ is a CG-group.
\end{cor}

\begin{proof}
In view of \cite[Theorem 3.2.14]{perumal}, $\frac{G}{Z(G)}$ is a Frobenius group with cyclic complement and therefore, the result follows from Corollary \ref{CG17cor}. 
\end{proof}

\begin{cor}\label{CG31}
Let $G$ be a finite group such that $\frac{G}{Z(G)}$ is a Frobenius group with cyclic kernel. Then $G$ is a CG-group. 
\end{cor}

\begin{proof}
In view of \cite{Fcyclic},  $\frac{G}{Z(G)}$ has cyclic Frobenius complement and therefore the result follows from Proposition \ref{CG17}.
\end{proof}

In this connection we would like to mention the following three lemmas:

\begin{lem}\label{CG18}
Let $G$ be a finite group such that $\frac{G}{Z(G)}=\frac{K}{Z(G)} \rtimes \frac{H}{Z(G)}$ is a Frobenius group with $K$ and $H$ abelian. Then $\mid G' \cap Z(G) \mid=1$.
\end{lem}

\begin{proof}
In the present scenario,  from Proposition \ref{CG17} and \cite[Theorem 2.2]{herzog}, we have $\mid G' \mid= \mid \frac{K}{Z(G)}\mid=\mid(\frac{G}{Z(G)})'\mid$ and hence the result follows.    
\end{proof}

\begin{lem}\label{CG181}
Let $G$ be a finite group such that $\frac{G}{Z(G)}$ is a Frobenius group with cyclic kernel. Then $\mid G' \cap Z(G) \mid=1$.
\end{lem}

\begin{proof}
In view of \cite{Fcyclic},  $\frac{G}{Z(G)}$ has cyclic Frobenius complement and therefore the result follows from Lemma \ref{CG18}.  
\end{proof}

\begin{lem}\label{CGcor18}
Let $G$ be a finite group such that $\frac{G}{Z(G)}=\frac{K}{Z(G)} \rtimes \frac{H}{Z(G)}$ is a Frobenius group with  $H$ abelian. If $G'$ is abelian, then $\mid G' \cap Z(G) \mid=1$.
\end{lem}

\begin{proof}
Using arguments similar to Corolary \ref{CG17cor} and Lemma  \ref{CG18} we get the result.  
\end{proof}

As a consequence of Proposition \ref{CG17}, we have the following three results:

\begin{prop}\label{CG112}
Let $G$ be a finite group such that $G'$ is of prime order $p$ and $G' \cap Z(G)=\lbrace 1 \rbrace$. Then  $G$ is  a CG-group.
\end{prop}

\begin{proof}
In the present scenario, we have  $(\frac{G}{Z(G)})'$ is of prime order $p$ and $Z(\frac{G}{Z(G)})$ is of order $1$. Therefore in view of \cite[Proposition 5]{dR79}, $\frac{G}{Z(G)}$ is a Frobenius group with cyclic kernel and cyclic complement. Therefore $G$ is a CG-group by Corollary \ref{CG31}.
\end{proof}

\begin{prop}\label{CG11}
If $G$ is a finite group such that $\mid \frac{G}{Z(G)} \mid =pqr$ or $p^2q$ for any primes $p<q<r$, then $G$ is a CG-group.
\end{prop}

\begin{proof}
Suppose $\mid \frac{G}{Z(G)} \mid =pqr$. By \cite[Proposition 2.5]{baishya2}, we have  $\mid Z(\frac{G}{Z(G)}) \mid =1$ and therefore in view of Remark \ref{rem1}, $\frac{G}{Z(G)}$ is a solvable CA-group. But then using Remark \ref{Remark3}, we have $\frac{G}{Z(G)}$ is a Frobenius group with cyclic kernel and cyclic complement. Now, the result follows using Proposition \ref{CG17}.

Next, suppose $\mid \frac{G}{Z(G)} \mid =p^2q$. For $q=3$, we have $\frac{G}{Z(G)} \cong A_4$ or $D_{12}$. If $\frac{G}{Z(G)} \cong A_4$, then
$\frac{G}{Z(G)}  \cong \frac{K}{Z(G)} \rtimes \frac{H}{Z(G)}$ is a Frobenius group with kernel of order $4$ and complement of order $3$. In the present scenario, we have $K$ and $H$ both are abelian, and therefore by Proposition \ref{CG17},  $G$ is a CG-group. On the otherhand if $\frac{G}{Z(G)} \cong D_{12}$, then by \cite[Proposition 2.8, Proposition 2.9]{baishya}, $G$ is a CG-group.

Now, we consider the case when $q>3$. Using \cite[Proposition 2.11]{baishya2}, we have
 $\mid Z(\frac{ G}{Z(G)})\mid=1$ or $p$. Now, If $\mid Z(\frac{G}{Z(G)}) \mid =p$, then by \cite[Proposition 2.11]{baishya2},  $\frac{G}{Z(G)} \cong C_{pq} \rtimes C_p$ and so $G$ is a CG-group by \cite[Proposition 2.9]{baishya}.

Again, if $\mid Z(\frac{G}{Z(G)}) \mid =1$, then in view of Remark \ref{rem1} and Remark \ref{Remark3}, $ \frac{G}{Z(G)}$ is a Frobenius group with cyclic kernel and  cyclic  complement of order $p^2$. Now, the result follows using Corollary \ref{CG31}. 
\end{proof}

\begin{prop}\label{CG14}
If $G$ is a finite group such that $\mid \frac{G}{Z(G)} \mid =pqr $ or $pq^2$,  $p<q<r$ be primes, then $\mid G' \cap Z(G) \mid=1$.
\end{prop}

\begin{proof}
By \cite[Proposition 2.5, Proposition 2.11]{baishya2}, we have $\mid Z(\frac{G}{Z(G)})\mid=1$ and hence in view of Remark \ref{rem1} and Remark \ref{Remark3}, $\frac{G}{Z(G)} \cong \frac{K}{Z(G)} \rtimes \frac{H}{Z(G)}$ is a Frobenius group with abelian kernel and cyclic complement.  In the present scenario, one can verify that $K$ and $H$ both are abelian. Therefore the result follows using Lemma \ref{CG18}. 
\end{proof}

For finite groups with central quotient of order $p^3$, $p$ being prime, we have the following result:

\begin{prop}\label{CG15}
Let $G$ be a finite group such that $\mid \frac{G}{Z(G)} \mid =p^3$, $p$ being prime. Then $G$ is a CG-group if and only if $G$ has an abelian normal subgroup of prime index.
\end{prop}

\begin{proof}
Let $G$ be a CG-group. In view of Remark \ref{rem1}, we have $G$ is a CA-group and  $C(x) \cap C(y) = Z(G)$ for any $x,y \in G \setminus Z(G), xy \neq yx$. Now, suppose $G$ has no centralizer of index $p$. Then $\mid C(x) \mid =\frac{\mid G \mid}{p^2}$ for all $x \in G \setminus Z(G)$ and consequently by \cite{ito}, $G=A \times P$, where $A$ is an abelian group and $P$ is a $p$-group. In the present scenario,  each proper centralizer of $G$ will contain exactly $p$ distinct right coset of $Z(G)$. Therefore $\mid Cent(G) \mid= p^2+p+2$. But then $\mid G' \mid=p(p+1)$, which is impossible. Therefore $G$ must have a centralizer of prime index $p$, say, $C(y)$ for some $y \in G \setminus Z(G)$.  Clearly, in view of \cite[Exercise 3 (p. 9)]{kur}, $G$ cannot have another centralizer of index $p$ and consequently $C(y) \lhd G$. Thus $G$ has an abelian normal subgroup of prime index.

Conversely, if $G$ has an abelian normal subgroup of prime index, then by Proposition \ref{CG7}, $G$ is a CG-group.
\end{proof}

As an application of Proposition \ref{CG17}, we also have the following result:

\begin{prop}\label{CG9}
Let $G$ be a finite non-abelian solvable group. If $N_G(H)=H$ for all non-abelian $H \leq G$, then $G$ is a CG-group. 
\end{prop}

\begin{proof}
If $G$ is a nilpotent group, then in view of  \cite[Proposition 2.6]{niketora}, $G$ is a minimal non-abelian $p$-group for some prime $p$ and therefore by Proposition \ref{CG6},  $G$ is a CG-group. 

Next, suppose $G$ is non-nilpotent. Then by \cite[Theorem 2.13]{niketora}, we have $\frac{G}{Z(G)}$ is a Frobenius group with complement of prime order $p$. Moreover, by \cite[Proposition 2.2]{niketora}, we have $G'$ is abelian. Therefore by Corollary \ref{CG17cor}, $G$ is a CG-group.
\end{proof}

It may be mentioned here that all the examples of CG-groups given in the earlier section are CA-groups. However, there exists CG-groups which are not CA-groups. For example, it can be verify that $S_4$ is a CG-group but not a CA-group. In this connection, we give the following result on CG-groups whose central quotient is $S_4$.

\begin{prop}\label{CG12}
Let $G$ be a finite group such that $\frac{G}{Z(G)} \cong S_4$. Then $G$ is a CG-group if and only if $G' \cong A_4$.
\end{prop}

\begin{proof}
If $G' \cong A_4$, then we have $\mid G' \cap Z(G) \mid=1$ noting that $\mid(\frac{G}{Z(G)})'\mid=\mid \frac{G'}{G' \cap Z(G)} \mid =\mid A_4 \mid$. Therefore by \cite[Lemma 3.1]{en09}, $\mid Cent(G) \mid= \mid Cent(\frac{G}{Z(G)}) \mid=\mid G' \mid+2$.

Conversely, suppose $\mid Cent(G) \mid=\mid G' \mid+2$. In view of \cite[Theorem 1.3]{rostami}, we have $\mid Cent(G) \mid= \mid Cent(\frac{G}{Z(G)}) \mid$ and consequently, $\mid G' \mid=\mid   (\frac{G}{Z(G)})'\mid=\mid \frac{G'}{G' \cap Z(G)} \mid$, forcing $\mid G' \cap Z(G) \mid=1$. Therefore we have $A_4 \cong (\frac{G}{Z(G)})'=G'$.
\end{proof}

\begin{rem}\label{CGREM}
Recall that a minimal non-nilpotent group is a non-nilpotent group all of whose proper subgroups are nilpotent. If $G$ is a finite minimal non-nilpotent group then $G$ is a solvable group of order $p^mq^n$ ($p, q$ are distinct primes) with a unique Sylow $p$-subgroup $P$ and a cyclic Sylow $q$-subgroup $Q$. Moreover, we have $G'=P$ (see \cite{tar}). 
\end{rem}

\begin{prop}\label{CG19}
Let $G$ be a finite minimal non-nilpotent group.
\begin{enumerate}
	\item  If $G'$ is abelian, then $G$ is a CG-group.
	\item  $\frac{G}{Z(G)}$ is a CG-group.
\end{enumerate}. 
\end{prop}

\begin{proof}

a) By Remark \ref{CGREM}, we have $G'=P$. Again, by \cite[Lemma 2.3]{amiri17}, we have $\frac{G}{Z(G)} = \frac{PZ(G)}{Z(G)} \rtimes \frac{QZ(G)}{Z(G)}$ is a Frobenius group with cyclic complement. Now, if $G'$ is abelian, then $PZ(G)$ will be abelian and hence $G$ is a CG-group by Proposition \ref{CG17}.\\

b) Let $G$ be a minimal non-nilpotent group. By \cite[Lemma 2.3]{amiri17},  $\frac{G}{Z(G)}$ is a Frobenius group with elementary abelian kernel $\frac{K}{Z(G)}$ and complement of prime order. Again, by \cite[Theorem 2.2]{herzog}, $(\frac{G}{Z(G)})' =\frac{K}{Z(G)}$. Therefore from the definition of Frobenius group, $\frac{G}{Z(G)}$ is a CG-group.
\end{proof}

\begin{prop}\label{CG203}
Let $G$ be a finite group such that $\frac{G}{Z(G)}$ is minimal non-nilpotent. If $G'$ is abelian, then $G$ is a CG-group.
\end{prop}

\begin{proof}
In view of Remark \ref{CGREM}, we have $\frac{G}{Z(G)} = \frac{G'Z(G)}{Z(G)} \rtimes \frac{Q}{Z(G)}$, where $Q$ is a subgroup of $G$ such that $\frac{Q}{Z(G)}$ is cyclic.  Now, if $G'$ is abelian, then $G'Z(G)$ will be abelian and consequently, by third isomprphic theorem, we have  $\frac{G}{G'Z(G)} \cong \frac{Q}{Z(G)}$. Therefore in view of Proposition \ref{isaacs}, we have $\mid G'Z(G) \mid= \mid G' \mid \mid G'Z(G) \cap Z(G) \mid = \mid G' \mid \mid Z(G) \mid$ and hence $\mid G' \cap Z(G)\mid=1$. Therefore using \cite[Lemma 3.1]{en09} and Proposition \ref{CG19}, we have 
$\mid Cent(G) \mid= \mid Cent(\frac{G}{Z(G)}) \mid= \mid (\frac{G}{Z(G)})'\mid +2=\mid G' \mid+2$.
\end{proof}

\begin{prop}\label{CG23}
Let $G$ be a finite group such that $\frac{G}{Z(G)}$ is minimal non-abelian which is not a $p$-group, $p$ being prime. If $G'$ is abelian, then $G$ is a CG-group.
\end{prop}

\begin{proof}
In the present scenario, in view of \cite[Theorem 1.1]{rostami}, $\frac{G}{Z(G)}$ is a Frobenius group with cyclic complement. Now, the result follows using Corollary \ref{CG17cor}.
\end{proof}

Finally, we conclude the paper with a counterexample to the following conjecture \cite[Conjecture 2.3]{con} which is also a counterexample to \cite[Question 1.2]{con} :

\begin{conj}\label{conj}
Let $G$ and $S$ be  finite groups. Is it true that if  $\mid \Cent(G) \mid= \mid \Cent(S) \mid$ 
and $\mid G' \mid= \mid S' \mid$, then $G$ is isoclinic to $S$?   
\end{conj}

Consider any non-abelian group $G$ of order $27$ and $S_3$. In view of \cite[Corollary 2.5]{baishya}, we have $\mid \Cent(G) \mid= \mid \Cent(S_3) \mid=5$ and $\mid G' \mid= \mid {S_3}' \mid=3$. But $G$ is not isoclinic to $S_3$.


\end{document}